\documentclass[12pt]{amsart}

\usepackage{amsmath,amsfonts,amssymb,amsthm,mathabx,shuffle,latexsym,mathtools,stmaryrd}
\usepackage{enumitem}
\usepackage[usenames, dvipsnames, svgnames]{xcolor}
\usepackage[backref=page]{hyperref}
\usepackage{ytableau}
\usepackage{tikz-cd}
\usepackage{pbox}

\usepackage{tikz}
\usetikzlibrary{calc, shapes, backgrounds,arrows,positioning,backgrounds}
\tikzset{>=stealth',
  head/.style = {fill = white, text=black},
  plaque/.style = {draw, rectangle, minimum size = 9mm, fill=Gainsboro}, 
   Kplaque/.style = {draw, rectangle, minimum size = 9mm, fill=white}, 
  newplaque/.style = {draw=red, ellipse, minimum size = 9mm,ultra thick, fill=white}, 
  posex/.style={->,thick},
  lift/.style={right hook->},
  beta/.style={dashed, <->},
  pil/.style={->,thick},
  junct/.style = {draw,circle,inner sep=0.5pt,outer sep=0pt, fill=black}
  }

\newtheorem{theorem}{Theorem}[section]
\newtheorem{lemma}[theorem]{Lemma}
\newtheorem{proposition}[theorem]{Proposition}

\theoremstyle{definition}

\newtheorem{example}[theorem]{Example}

\theoremstyle{remark}
\newtheorem{remark}[theorem]{Remark}

\numberwithin{equation}{section}

\newcommand{\Des}{\ensuremath{\mathrm{Des}}}

\newcommand{\Sym}{\ensuremath{\mathrm{Sym}}}
\newcommand{\QSym}{\ensuremath{\mathrm{QSym}}}
\newcommand{\NSym}{\ensuremath{\mathrm{NSym}}}

\newcommand{\eS}{\ensuremath{\mathcal{E}}}

\newcommand{\SET}{\ensuremath{\mathrm{SET}}}
\newcommand{\NSET}{\ensuremath{\mathrm{NSET}}}

\newcommand{\StdTab}{\ensuremath{\mathrm{SRIT}}}

\newcommand{\C}{\ensuremath{\mathbb{C}}}

\newcommand{\Set}{\ensuremath{\mathbb{S}}}

\newcommand{\excise}[1]{}

\newlength\cellsize 

\savebox2{%
\begin{picture}(18,18)
\put(0,0){\line(1,0){18}}
\put(0,0){\line(0,1){18}}
\put(18,0){\line(0,1){18}}
\put(0,18){\line(1,0){18}}
\end{picture}}

\newcommand\cellify[1]{\def\thearg{#1}\def\nothing{}%
\ifx\thearg\nothing\vrule width0pt height\cellsize depth0pt%
  \else\hbox to 0pt{\usebox2\hss}\fi%
  \vbox to 18\unitlength{\vss\hbox to 18\unitlength{\hss$#1$\hss}\vss}}

\newcommand\tableau[1]{\vtop{\let\\=\cr
\setlength\baselineskip{-12000pt}
\setlength\lineskiplimit{12000pt}
\setlength\lineskip{0pt}
\halign{&\cellify{##}\cr#1\crcr}}}

\usepackage[colorinlistoftodos]{todonotes}

\ytableausetup{centertableaux}

\begin{document}


\title{Indecomposable $0$-Hecke modules for extended Schur functions}  

\author[D. Searles]{Dominic Searles}
\address[DS]{Department of Mathematics and Statistics, University of Otago, Dunedin 9016, New Zealand}
\email{dominic.searles@otago.ac.nz}


\date{June 11, 2019}



\begin{abstract}
The extended Schur functions form a basis of quasisymmetric functions that contains the Schur functions. 
We provide a representation-theoretic interpretation of this basis by constructing $0$-Hecke modules whose quasisymmetric characteristics are the extended Schur functions. We further prove these modules are indecomposable. 
\end{abstract}

\maketitle

%
\section{Introduction}
%
\label{sec:introduction}
The Schur basis of the algebra of symmetric functions $\Sym$ has applications in wide-ranging areas of mathematics, including representation theory of the symmetric group and the geometry of Grassmannians. Important generalizations of $\Sym$ include the quasisymmetric functions $\QSym$ and the noncommutative symmetric functions $\NSym$, which are dual to one another as Hopf algebras. The symmetric functions $\Sym$ can be realised as a subalgebra of $\QSym$ and as a quotient algebra of $\NSym$.

There has been significant recent interest in constructing bases of $\QSym$ and $\NSym$ that generalize or share properties with the Schur functions. Central and well-studied examples include the \emph{quasisymmetric Schur}  \cite{HLMvW11:QS} and \emph{dual immaculate} \cite{BBSSZ:2} bases of $\QSym$, and their dual bases in $\NSym$: the \emph{noncommutative Schur} \cite{BLvW} and \emph{immaculate} bases \cite{BBSSZ:2}, respectively.

The \emph{extended Schur} basis of $\QSym$ was defined in \cite{Assaf.Searles:3} as the stable limits of polynomials arising from application of Kohnert's algorithm \cite{Kohnert} to certain cell diagrams. The nomenclature comes from the fact the extended Schur basis contains the Schur basis of $\Sym$, thus extends the Schur basis to a basis of $\QSym$. It was proved in \cite{Assaf.Searles:3} that extended Schur functions expand positively in the fundamental basis \cite{Gessel} of $\QSym$; dual immaculate and quasisymmetric Schur functions also expand positively in this basis. 
The dual basis to the extended Schur functions is known as the \emph{shin} basis of $\NSym$, introduced in \cite{CFLSX}. The noncommutative Schur, immaculate and shin bases are described in \cite{Campbell} as the three \emph{canonical Schur-like} bases of $\NSym$.

The interpretation of Schur functions as characters of irreducible representations of the symmetric group raises a natural question about potential representation-theoretic interpretations of generalisations of the Schur basis. Recently, in \cite{BBSSZ}, modules of the \emph{$0$-Hecke algebra} were constructed whose \emph{quasisymmetric characteristics} \cite{DKLT} are exactly the dual immaculate quasisymmetric functions. Then in \cite{TvW:1}, $0$-Hecke modules were constructed whose quasisymmetric characteristics are exactly the quasisymmetric Schur functions. In addition, $0$-Hecke modules were constructed in \cite{TvW:1} for the \emph{skew quasisymmetric Schur functions} of \cite{BLvW}, and $0$-Hecke actions on a family of tableaux related to the generalized Demazure atoms of \cite{HMR} were defined in \cite{TvW:2}. Our motivation is to complete the picture for the canonical Schur-like bases by providing a representation-theoretic interpretation of the extended Schur functions.

In this paper, we accomplish this by constructing a $0$-Hecke action on \emph{standard extended tableaux}, and proving that the quasisymmetric characteristics of the corresponding modules are exactly the extended Schur functions. 
We additionally prove these modules are indecomposable. In comparison, the modules for the dual immaculate quasisymmetric functions are also indecomposable \cite{BBSSZ}, while the modules for the quasisymmetric Schur functions are not in general: a direct sum decomposition is given in \cite{TvW:1} whose components are proved to be indecomposable in \cite{Koenig}.

\section{Background}

\subsection{Quasisymmetric functions and noncommutative symmetric functions}

A \emph{composition} is a finite sequence $\alpha = (\alpha_1, \ldots , \alpha_k)$ of positive integers. We call $\alpha_1,\ldots , \alpha_k$ the \emph{parts} of $\alpha$, and when $\alpha$ has $k$ parts we say the \emph{length} $\ell(\alpha)$ of $\alpha$ is $k$. When the parts of  $\alpha$ sum to $n$, we say that $\alpha$ is a composition of $n$, written $\alpha \vDash n$. Compositions of $n$ are in bijection with subsets of $[n-1] = \{1,\ldots , n-1\}$; given a composition $\alpha = (\alpha_1, \ldots , \alpha_k)\vDash n$, define $\Set(\alpha)$ to be the subset $\{\alpha_1, \alpha_1+\alpha_2, \ldots , \alpha_1+\cdots + \alpha_{k-1}\}$ of $[n-1]$. We say a composition $\beta$ \emph{refines} a composition $\alpha$ if $\alpha$ can be obtained by summing consecutive entries of $\beta$.

\begin{example}
Let  $\alpha = (2,1,3) \vDash 6$. Then $\Set(\alpha) = \{2,3\}\subset [5]$. The composition $(2,1,3)$ refines the composition $(2,4)$ but does not refine $(4,2)$.
\end{example}

Let $\C[[x_1, x_2, \dots ]]$ denote the Hopf algebra of formal power series of bounded degree in infinitely many commuting variables. The Hopf algebra $\QSym$ of quasisymmetric functions \cite{Gessel} is the subalgebra of $\C[[x_1, x_2, \dots ]]$ consisting of those formal power series $f$ such that for every composition $\alpha = (\alpha_1, \ldots \alpha_k)$,
\[
  [x_{i_1}^{\alpha_1} \cdots x_{i_{k}}^{\alpha_{k}} \mid  f] = [x_{j_1}^{\alpha_1} \cdots x_{j_{k}}^{\alpha_{k}} \mid  f] 
\]
for any two sequences $1 \leq i_1< \cdots < i_{k}$ and $1 \leq j_1< \cdots < j_{k}$, where $[x_{i_1}^{\alpha_1} \cdots x_{i_{k}}^{\alpha_{k}} \mid  f]$ is the coefficient of the monomial $x_{i_1}^{\alpha_1} \cdots x_{i_{k}}^{\alpha_k}$ in the monomial expansion of $f$.

The \emph{monomial} and \emph{fundamental quasisymmetric functions} $M_\alpha$ and $F_\alpha$ are additive bases of $\QSym$ introduced in \cite{Gessel}. They are indexed by compositions, and defined by
\[M_\alpha = \sum_{i_1< i_2 < \cdots < i_k} x_{i_1}^{\alpha_1} \cdots x_{i_{k}}^{\alpha_{k}} \qquad \mbox{ and } \qquad F_\alpha = \sum_{\beta \text{ refines }\alpha} M_\beta.\]
\begin{example}
Let $\alpha = (1,3,1)$. We have
\[M_{(1,3,1)} = \sum_{i<j<k}x_ix_j^3x_k\]
and
\[F_{(1,3,1)} = M_{(1,3,1)} +  M_{(1,2,1,1)} + M_{(1,1,2,1)} + M_{(1,1,1,1,1)}.\]
\end{example}

The Hopf algebra $\NSym$ of noncommutative symmetric functions \cite{GKLLRT} is an analogue of the symmetric functions in noncommuting variables. It is generated by elements $H_1, H_2, \ldots$ with no relations, and has additive basis $\{H_\alpha\}$ indexed by compositions $\alpha=(\alpha_1, \ldots , \alpha_k)$, where the \emph{complete homogeneous function} $H_\alpha$ is defined to be the product $H_{\alpha_1}\cdots H_{\alpha_k}$. 
As Hopf algebras, $\NSym$ is dual to $\QSym$ via the pairing $\langle H_\alpha, M_\beta\rangle = \delta_{\alpha, \beta}$. The dual basis in $\NSym$ to the fundamental basis $\{F_\alpha\}$ of $\QSym$ is the \emph{ribbon Schur} functions $\{{\bf r}_\alpha\}$.

\subsection{Standard extended tableaux and extended Schur functions}

The diagram $D(\alpha)$ of a composition $\alpha$ is the array of boxes in the plane with $\alpha_i$ boxes in row $i$, left-justified. We depict composition diagrams in French notation, i.e., the bottom row is row $1$. 

\begin{example}
The diagram $D(\alpha)$ of $\alpha=(2,1,3)$ is shown below. 
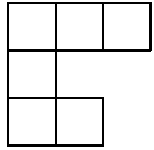
\[\tableau{ {\ } & {\ } & {\ } \\ {\ }  \\ {\ } & {\ } }\]
\end{example}

A \emph{standard extended tableau} \cite{Assaf.Searles:3} of shape $\alpha$ is a bijective assignment of the integers $\{1,2, \ldots n\}$ to the boxes of $D(\alpha)$, such that the entries in each row of $D(\alpha)$ increase from left to right and the entries in each column of $D(\alpha)$ increase from bottom to top. If $\alpha$ is a \emph{partition}, i.e., $\alpha_1\ge \alpha_2 \ge \cdots \ge \alpha_{\ell(\alpha)}$, then the standard extended tableaux of shape $\alpha$ are exactly the \emph{standard Young tableaux} of shape $\alpha$. We denote the collection of all standard extended tableaux of shape $\alpha$ by $\SET(\alpha)$.

\begin{remark}
The standard extended tableaux defined above are a vertical reflection of the standard extended tableaux defined in \cite{Assaf.Searles:3}, which are fillings of right-justified composition diagrams in which entries decrease from left to right along rows and decrease down columns. 
\end{remark}

\begin{example}\label{ex:SET}
The standard extended tableaux of shape $(2,1,3)$ are shown below.
  \begin{displaymath}
   \begin{array}{c@{\hskip2\cellsize}c@{\hskip2\cellsize}c@{\hskip2\cellsize}c}
T_1 = \tableau{ 4 & 5 & 6 \\ 3  \\ 1 & 2 } & T_2 = \tableau{ 4 & 5 & 6 \\ 2  \\ 1 & 3 } & T_3 = \tableau{ 3 & 5 & 6 \\ 2  \\ 1 & 4 }
\end{array}
  \end{displaymath}
\end{example}

We say an entry $i$ of a standard extended tableau $T$ is a \emph{descent} of $T$ if $i$ is weakly to the right of $i+1$ in $T$. Define the \emph{descent composition} $\Des(T)$ of $T$ to be the composition $\alpha$ such that $\Set(\alpha)$ is the set of all descents of $T$. 

\begin{example}\label{ex:Des}
Consider the three standard extended tableaux from Example~\ref{ex:SET}. The descents of $T_1$ are $2$ and $3$, the descents of $T_2$ are $1$ and $3$, and the descents of $T_3$ are $1$, $2$ and $4$. Hence $\Des(T_1) = (2,1,3)$, $\Des(T_2) = (1,2,3)$ and $\Des(T_3) = (1,1,2,2)$. 
\end{example}

Let $\alpha$ be a composition. In \cite{Assaf.Searles:3}, the \emph{extended Schur functions} $\eS_\alpha$ were defined as the stable limits of polynomials obtained by applying Kohnert's algorithm \cite{Kohnert} to right-justified cell diagrams. The extended Schur functions are quasisymmetric and in fact expand positively in the fundamental basis of $\QSym$ \cite{Assaf.Searles:3}. We take the formula for this expansion as definitional for the extended Schur functions.

\begin{theorem}\label{thm:eStofund}\cite{Assaf.Searles:3}
Let $\alpha$ be a composition. Then
\[\eS_\alpha = \sum_{T\in \SET(\alpha)}F_{\Des(T)}.\]
\end{theorem}

\begin{example}
By Examples~\ref{ex:SET} and \ref{ex:Des}, we have
\[\eS_{(2,1,3)} = F_{(2,1,3)} + F_{(1,2,3)} + F_{(1,1,2,2)}.\]
\end{example}

\begin{theorem}\cite{Assaf.Searles:3}
The extended Schur functions $\{\eS_\alpha\}$ form a basis of $\QSym$.
\end{theorem}

Every Schur function is in fact an extended Schur function. We may take the following result as definitional for the celebrated Schur functions:

\begin{proposition}\cite{Assaf.Searles:3}
If $\alpha$ is a partition, then the extended Schur function $\eS_\alpha$ is equal to the Schur function $s_\alpha$. 
\end{proposition}

The extended Schur functions are thus a basis of $\QSym$ that contains the Schur basis of symmetric functions. We note that other important and well-studied bases of $\QSym$ such as the quasisymmetric Schur functions, fundamental quasisymmetric functions and dual immaculate quasisymmetric functions do not contain the Schur functions.

The extended Schur functions are dual to the \emph{shin} basis $\eS^*_\alpha$ of noncommutative symmetric functions introduced and studied in \cite{CFLSX}. The shin functions have the property that the image of $\eS^*_\alpha$ under the natural projection from $\NSym$ to $\Sym$ is the Schur function $s_\alpha$ if $\alpha$ is a partition, and $0$ otherwise. Complete homogeneous functions expand positively in the shin basis \cite{CFLSX}, which then implies via duality that extended Schur functions expand positively into the monomial basis of quasisymmetric functions. 

Since extended Schur functions expand positively into the fundamental basis of quasisymmetric functions (Theorem~\ref{thm:eStofund}), duality implies the following result for shin functions.

\begin{proposition}
The ribbon Schur functions expand positively in the shin basis of $\NSym$ via the formula
\[{\bf r}_\beta = \sum_\beta K_{\alpha, \beta} \eS^*_\alpha\]
where $K_{\alpha,\beta}$ is the number of $T\in \SET(\alpha)$ such that $\Des(T)=\beta$.
\end{proposition}
\begin{proof}
By Theorem~\ref{thm:eStofund} and the definition of $K_{\alpha,\beta}$, we have
\[\eS_\alpha = \sum_\beta K_{\alpha,\beta}F_\beta.\]
Hence, by the fact the ribbon Schur functions are dual to the fundamental quasisymmetric functions, we have 
\[\langle \eS_\alpha , {\bf r}_\beta \rangle = K_{\alpha,\beta}.\]
Therefore, since the shin functions are dual to the extended Schur functions, we have
\[{\bf r}_\beta = \sum_\beta K_{\alpha, \beta} \eS^*_\alpha.\]
\end{proof}

\subsection{$0$-Hecke algebras and quasisymmetric characteristic}

The $0$-Hecke algebra $H_n(0)$ is defined to be the algebra over $\C$ with generators $T_1, \ldots , T_{n-1}$ subject to relations
\begin{align*}
T_i^2 & = T_i \quad \mbox{ for all } 1\le i \le n-1 \\
T_iT_j & = T_jT_i \quad \mbox{ for all } i, j \mbox{ with } |i-j|\ge 2 \\
T_iT_{i+1}T_i & =T_{i+1}T_iT_{i+1} \quad \mbox{ for all } 1\le i \le n-2.
\end{align*}

For any permutation $\sigma \in S_n$, one can define an element $T_\sigma \in H_n(0)$ by $T_\sigma = T_{i_1}T_{i_2} \cdots T_{i_r}$ where $s_{i_1} s_{i_2} \cdots s_{i_r}$ is any reduced word for $\sigma$. Then $\{T_\sigma : \sigma \in S_n\}$ is an additive basis for $H_n(0)$.

The \emph{Grothendieck group} $\mathcal{G}_0(H_n(0))$ is the linear span of the isomorphism classes of the finite-dimensional representations of $H_n(0)$, subject to the relation $[Y]=[X]+[Z]$ whenever one has a short exact sequence $0\rightarrow X\rightarrow Y\rightarrow Z\rightarrow 0$ of $H_n(0)$-representations $X,Y,Z$.

There are $2^{n-1}$ irreducible representations of $H_n(0)$; these may be indexed by the $2^{n-1}$ compositions of $n$. Let $\mathcal{F}_\alpha$ denote the irreducible representation corresponding to the composition $\alpha$. By \cite{Norton}, $\mathcal{F}_\alpha$ is one-dimensional, hence equal to the span of some nonzero vector $v_\alpha$. The structure of $\mathcal{F}_\alpha$ as a $H_n(0)$-representation is given by the following action of the generators $T_i$ of $H_n(0)$:

\begin{align}\label{eq:irreps}
T_i(v_\alpha) = \begin{cases} v_\alpha & \mbox{ if } i\notin \Set(\alpha) \\
					       0 & \mbox{ if } i \in \Set(\alpha).
					       \end{cases}
 \end{align}
Define
\[\mathcal{G} = \bigoplus_{n\ge 0} \mathcal{G}_0(H_n(0)).\]
The set $\{\mathcal{F}_\alpha\}$ as $\alpha$ ranges over all compositions is an additive basis of $\mathcal{G}$. Moreover, $\mathcal{G}$ has a ring structure via the induction product. There is a ring isomorphism $ch:\mathcal{G}\rightarrow \QSym$ \cite{DKLT} defined by setting $ch([\mathcal{F}_\alpha]) = F_\alpha$. For any $H_n(0)$-module $X$, the image $ch([X])$ is called the \emph{quasisymmetric characteristic} of $X$.

\section{Modules for extended Schur functions}

The immaculate basis, noncommutative Schur basis and the shin basis have been described as the \emph{canonical Schur-like} bases of $\NSym$ \cite{Campbell}. Interpretations of the dual bases of the first two as quasisymmetric characteristics of certain $H_n(0)$-modules are given in \cite{BBSSZ} and \cite{TvW:1} respectively. We complete this picture by constructing $H_n(0)$-modules whose quasisymmetric characteristics are the extended Schur functions.

Specifically, in this section we construct a $H_n(0)$-module $X_\alpha$ for each composition $\alpha$ of $n$, and prove that the quasisymmetric characteristic $ch([X_\alpha])$ is equal to the extended Schur function $\eS_\alpha$. Additionally, we prove that these modules $X_\alpha$ are indecomposable for all compositions $\alpha$.

\subsection{$0$-Hecke actions and modules}

Given a composition $\alpha$ of $n$, define a \emph{standard row-increasing tableau} of shape $\alpha$ to be a bijective assignment of the integers $1, \ldots , n$ to the boxes of $D(\alpha)$ such that entries increase from left to right along rows. We note that no condition is imposed on columns. Let $\StdTab(\alpha)$ denote the set of standard row-increasing tableaux of shape $\alpha$.
For $T\in \StdTab(\alpha)$ and $1\le i \le n-1$, define 
\[\pi_i(T) = \begin{cases} T & \mbox{ if } i \mbox{ is weakly above } i+1 \mbox{ in } T \\
				s_i(T) & \mbox{ otherwise }\end{cases}
\]
where $s_i(T)$ denotes the filling of $D(\alpha)$ obtained from $T$ by swapping the entries $i$ and $i+1$.

\begin{example}
Let $\alpha = (4,2,3)$ and let 
\[T = \tableau{ 4 & 6 & 7 \\ 1 & 5  \\ 2 & 3 & 8 & 9 }\in \StdTab(\alpha).\]
Then $\pi_4(T) = \pi_8(T)=T$, while 
\[\pi_5(T) = s_5(T) =  \tableau{ 4 & 5 & 7 \\ 1 & 6  \\ 2 & 3 & 8 & 9 }\in \StdTab(\alpha).\]
\end{example}

 Let $V_\alpha$ denote the $\mathbb{C}$-vector space spanned by $\StdTab(\alpha)$.

\begin{proposition}\label{prop:Heckeaction}				
The operators $\pi_i$ define a $H_n(0)$-action on $V_\alpha$. Specifically, we have $\pi_i(T)\in V_\alpha$ for all $T\in V_\alpha$ and all $1\le i \le n-1$, and the $\pi_i$ satisfy the relations for the generators $T_i$ of the $0$-Hecke algebra.
\end{proposition}
\begin{proof}
Let $T\in \StdTab(\alpha)$. First we note that $\pi_i(T)\in V_\alpha$, since $\pi_i$ can exchange the entries $i$ and $i+1$ only if they are in different rows, in which case exchanging $i$ and $i+1$ does not affect the relative order of the entries in either of the two rows containing $i$ or $i+1$.

If $i$ is weakly above $i+1$ in $T$, then $\pi_i(T)=T$ so $\pi_i^2(T)=T = \pi_i(T)$. Otherwise, $\pi_i(T)=s_i(T)$, and then $\pi_i^2(T)=\pi_i(s_i(T)) = s_i(T)=\pi_i(T)$. Hence $\pi_i^2=\pi_i$.

If $|i-j|\ge 2$, then $\{i,i+1\}\cap\{j,j+1\} = \emptyset$, so $\pi_i$ and $\pi_j$ affect disjoint pairs of boxes and thus it is clear that $\pi_i\pi_j(T) = \pi_j\pi_i(T)$.

Finally, we show $\pi_i\pi_{i+1}\pi_i = \pi_{i+1}\pi_i\pi_{i+1}$. We check the following cases:
\begin{enumerate}
\item $i$ is weakly above $i+1$; $i+1$ is weakly above $i+2$
\item $i$ is strictly below $i+1$; $i+1$ is strictly below $i+2$
\item $i$ is weakly above $i+1$; $i+1$ is strictly below $i+2$
 \begin{enumerate}
 \item $i$ is weakly above $i+2$
 \item $i$ is strictly below $i+2$
 \end{enumerate}
\item $i$ is strictly below $i+1$; $i+1$ is weakly above $i+2$
 \begin{enumerate}
 \item $i$ is weakly above $i+2$
 \item $i$ is strictly below $i+2$
 \end{enumerate}
\end{enumerate}
\noindent
(1): Here we have $\pi_i(T)=\pi_{i+1}(T)=T$, hence $\pi_i\pi_{i+1}\pi_i(T) = \pi_{i+1}\pi_i\pi_{i+1}(T)$.

\noindent
(2): Here it is straightforward to check $\pi_i\pi_{i+1}\pi_i(T) = s_is_{i+1}s_i(T) = s_{i+1}s_is_{i+1}(T) = \pi_{i+1}\pi_i\pi_{i+1}(T)$.

\noindent
(3): In this case, we have $\pi_i(T)=T$ and $\pi_{i+1}(T) = s_{i+1}(T)$. Hence $\pi_i\pi_{i+1}\pi_i(T) = \pi_is_{i+1}(T)$ and $\pi_{i+1}\pi_i\pi_{i+1}(T)= \pi_{i+1}\pi_is_{i+1}(T)$. Then we have:

(3a):  Here, $\pi_i(s_{i+1}(T)) = s_{i+1}(T)$. So $\pi_i\pi_{i+1}\pi_i(T) = s_{i+1}(T)$ and $\pi_{i+1}\pi_i\pi_{i+1}(T)= \pi_{i+1}s_{i+1}(T) = s_{i+1}(T)$ (since $s_{i+1}(T)$ has $i+1$ weakly above $i+2$).

(3b): Here, $\pi_i(s_{i+1}(T)) = s_is_{i+1}(T)$. So $\pi_i\pi_{i+1}\pi_i(T) = s_is_{i+1}(T)$ and $\pi_{i+1}\pi_i\pi_{i+1}(T)= \pi_{i+1}s_is_{i+1}(T)$. But $\pi_{i+1}s_is_{i+1}(T)= s_is_{i+1}(T)$; this is because $s_is_{i+1}$ sends $i+1$ to the original position of $i$ in $T$ and $i+2$ to the original position of $i+1$ in $T$, meaning that $i+1$ is weakly above $i+2$ in $s_is_{i+1}(T)$.

\noindent  
(4): In this case, we have $\pi_{i+1}(T) = T$ and $\pi_i(T) = s_i(T)$. Hence $\pi_i\pi_{i+1}\pi_i(T) = \pi_i\pi_{i+1}s_i(T)$ and $\pi_{i+1}\pi_i\pi_{i+1}(T)= \pi_{i+1}s_i(T)$. Then we have:

(4a): Here, $\pi_{i+1}s_i(T) = s_i(T)$. So $\pi_{i+1}\pi_i\pi_{i+1}(T)= s_i(T)$ and $\pi_i\pi_{i+1}\pi_i(T) = \pi_i s_i(T) = s_i(T)$ (since $s_i(T)$ has $i$ weakly above $i+1$).

4(b): Here $\pi_{i+1}s_i(T) = s_{i+1}s_i(T)$. So $\pi_{i+1}\pi_i\pi_{i+1}(T)= s_{i+1}s_i(T)$ and $\pi_i\pi_{i+1}\pi_i(T) = \pi_i s_{i+1}s_i(T)$. But $\pi_i s_{i+1}s_i(T)=s_{i+1}s_i(T)$; this is because $s_{i+1}s_i$ sends $i$ to the original position of $i+1$ in $T$ and $i+1$ to the original position of $i+2$ in $T$, meaning that $i$ is weakly above $i+1$ in $s_{i+1}s_i(T)$.
\end{proof}

\begin{remark}
This action is equivalent to the $H_n(0)$-action defined on words of content $\alpha$ in \cite{BBSSZ}. We prefer to work directly with tableaux of shape $D(\alpha)$, and include the proof of Proposition~\ref{prop:Heckeaction} above for completeness.
\end{remark}

Let $\NSET(\alpha)$ denote $\StdTab(\alpha)\setminus \SET(\alpha)$, i.e., those elements of $\StdTab(\alpha)$ in which entries do not increase up some column. Let $Y_\alpha$ denote the vector subspace of $V_\alpha$ spanned by $\NSET(\alpha)$.

\begin{lemma}
The vector space $Y_\alpha$ is an $H_n(0)$-submodule of $V_\alpha$.
\end{lemma}
\begin{proof}
Suppose $T\in \NSET(\alpha)$. Then $T$ has a pair of entries $j<k$ such that $j$ is above $k$ in the same column. If $k>j+1$, then for any $1\le i \le n-1$, $\pi_i$ can change only one of $j,k$, and by at most $1$, so $\pi_i(T)\in Y_\alpha$. If $k=j+1$, then $j$ is above $j+1=k$, so $\pi_j(T) = T \in Y_\alpha$. It remains to observe that $\pi_i$ for $i\neq j$ either has no effect on the boxes with entries $j$ or $j+1$, or it replaces $j$ with $j-1$ or $j+1$ with $j+2$, which does not change the relative order of the entries of these two boxes. Hence $\pi_i(T)\in \NSET(\alpha)$ for all $1\le i \le n-1$.
\end{proof}

Define $X_\alpha$ to be the quotient module $V_\alpha/Y_\alpha$. Then $\SET(\alpha)$ is a basis of $X_\alpha$.

\begin{theorem}\label{thm:actiononsetmodule}
For any $1\le i \le n-1$ and any composition $\alpha$ of $n$, the action of $\pi_i$ on $X_\alpha$ is given by 
\[\pi_i(T) = \begin{cases} T & \mbox{ if } i \mbox{ is strictly left of } i+1 \mbox{ in } T \\
				      0 & \mbox{ if } i \mbox{ and } i+1 \mbox{ are in the same column of } T \\
				      s_i(T) & \mbox{ if } i \mbox{ is strictly right of } i+1 \mbox{ in } T
				      \end{cases}
				      \]
				      for any $T\in \SET(\alpha)$.
\end{theorem}
\begin{proof}
Let $T\in \SET(\alpha)$. First suppose $i$ is strictly left of $i+1$ in $T$. Since entries increase in both rows and columns of $T$, $i$ cannot be strictly below $i+1$ in $T$; if it were, the entry in the row of $i+1$ and column of $i$ would have to be strictly larger than $i$ and strictly smaller than $i+1$, which is impossible. Hence $\pi_i(T)=T\in \SET(\alpha)$. 

Now suppose $i$ and $i+1$ are in the same column of $T$. Since $T\in \SET(\alpha)$, $i$ is strictly below $i+1$. Then $\pi_i(T) = s_i(T)$, in which $i$ is strictly above $i+1$. Hence $\pi_i(T)\in \NSET_\alpha$, i.e. $\pi_i(T) = 0$ in $X_\alpha = V_\alpha/Y_\alpha$. 

Finally suppose $i$ is strictly right of $i+1$ in $T$. Since entries increase along rows and up columns of $T$, $i$ cannot also be weakly above $i+1$ in $T$. Hence $\pi_i(T)=s_i(T)$. Since $i$ and $i+1$ are in different rows and different columns, $\pi_i(T)=s_i(T)\in \SET(\alpha)$.
\end{proof}

\begin{remark}
It is also possible, though tedious, to show directly that the operators $\pi_i$ on $\SET(\alpha)$ defined in Theorem~\ref{thm:actiononsetmodule} satisfy the $0$-Hecke relations. 
\end{remark}

\begin{example}
Let $\alpha = (4,2,3)$ and let 
\[T = \tableau{ 4 & 8 & 9 \\ 3 & 7  \\ 1 & 2 & 5 & 6 }\in \SET(\alpha).\]
Then $\pi_5(T)=T$, $\pi_7(T)=0$, and
\[\pi_6(T) = s_6(T) =  \tableau{ 4 & 8 & 9 \\ 3 & 6  \\ 1 & 2 & 5 & 7 }\in \SET(\alpha).\]
\end{example}

Define a relation $\preceq$ on $\SET(\alpha)$ by setting $S\preceq T$ if we can obtain $S$ from $T$ by applying a (possibly empty) sequence of the $\pi_i$ operators.

\begin{lemma}
The relation $\preceq$ is a partial order on $\SET(\alpha)$.
\end{lemma}
\begin{proof}
It is clear from the definition that $\preceq$ is reflexive and transitive. To see that it is antisymmetric, let $T\in \SET(\alpha)$ and define a vector $d_T$ by letting its $j$th entry $(d_T)_j$ be the sum of the entries in the first $j$ rows of $T$, for $1\le j \le \ell(\alpha)$. Suppose $\pi_i(T)=s_i(T)$. Then $i$ is strictly lower than $i+1$ in $T$ and $\pi_i$ exchanges $i$ and $i+1$. Consequently, $(d_{\pi_i(T)})_j\ge (d_T)_j $ for all $1\le j \le \ell(\alpha)$, and if $k$ is the index of the row in which $i$ appears in $T$, we have $(d_{\pi_i(T)})_k>(d_T)_k$. Therefore, if $S$ is obtained from $T$ via a sequence of the operators $\pi_i$, then either $S=T$ or there is some entry of $d_S$ that is strictly larger than the corresponding entry of $d_T$. Since application of the $\pi_i$ operators to $S$ cannot decrease entries of $d_S$, it is not possible to also obtain $T$ from $S$ via a sequence of these operators.
\end{proof}

Extend the partial order $\preceq$ on $\SET(\alpha)$ to a total order $\preceq^t$ arbitrarily. Suppose  $\preceq^t$ orders the $m$ elements of $\SET(\alpha)$ as $T_1 \preceq^t T_2 \preceq^t \cdots \preceq^t T_m$. For each $1\le j \le m$, define $X_j$ to be the $\mathbb{C}$-linear span of all $T_k\in \SET(\alpha)$ such that $k\le j$. It is immediate from the definitions of $\preceq$, $\preceq^t$ and $X_j$ that $X_j$ is an $H_n(0)$-module for each $1\le j\le m$.

We therefore have a filtration of $X_\alpha$ given by
\[0:=X_0\subset X_1 \subset X_2 \subset \cdots \subset X_m = X_\alpha.\]

By definition, each quotient module $X_j/X_{j-1}$ is one-dimensional and spanned by $T_j\in \SET(\alpha)$.

\begin{lemma}\label{lem:actiononquotientmodules}
For any $1\le i \le n-1$ and any $1\le j \le m$, we have
\[\pi_i(T_j) = \begin{cases}
                    T_j & \mbox{ if } i \mbox{ is strictly left of } i+1 \mbox{ in } T_j \\
                    0 & \mbox{ otherwise. }
                    \end{cases}
                    \]
\end{lemma}
\begin{proof}
If $i$ is strictly left of $i+1$ in $T_j$, then by Theorem~\ref{thm:actiononsetmodule} we have $\pi_i(T_j)=T_j$. If $i$ is not strictly left of $i+1$, then by Theorem~\ref{thm:actiononsetmodule} $\pi_i(T)$ is either $0$ or $s_i(T)$. However, $s_i(T) \in X_{j-1}$, so $s_i(T)=0$ in $X_j/X_{j-1}$.
\end{proof}

We may now prove our main result.

\begin{theorem}
Let $\alpha$ be a composition of $n$. The quasisymmetric characteristic of the $H_n(0)$-module $X_\alpha$ is the extended Schur function $\eS_\alpha$.
\end{theorem}				      
\begin{proof}
The quotient module $X_j/X_{j-1}$ is one-dimensional, thus irreducible. Lemma~\ref{lem:actiononquotientmodules} implies that 
\[\pi_i(T_j) = \begin{cases}
                    T_j & \mbox{ if } i \notin \Set(\Des(T_j)) \\
                    0 & \mbox{ if } i \in \Set(\Des(T_j)).
                    \end{cases}
                    \]
Therefore, by (\ref{eq:irreps}), $X_j/X_{j-1}$ is isomorphic as $H_n(0)$-modules to $\mathcal{F}_{\Des(T_j)}$. Hence we have $ch([X_j/X_{j-1}]) = F_{\Des(T_j)}$. Therefore,
\[ch([X_\alpha]) = \sum_{j=1}^m ch([X_j/X_{j-1}]) =  \sum_{j=1}^mF_{\Des(T_j)}= \sum_{T\in \SET(\alpha)}F_{\Des(T)} = \eS_\alpha,\]
where the last equality follows from Theorem~\ref{thm:eStofund}.
\end{proof}

\subsection{Indecomposability}

As is the case for the dual immaculate quasisymmetric functions, but not the case for the quasisymmetric Schur functions, the modules $X_\alpha$ for the extended Schur functions are indecomposable. We devote the remainder of the paper to establishing this fact, following the approach of \cite{BBSSZ} and \cite{TvW:1}.

Let $T^{{\rm sup}}_\alpha$ be the standard extended tableau of shape $\alpha$ whose entries in the $i$th row are the first $\alpha_i$ integers larger than $\alpha_1+\cdots + \alpha_{i-1}$. We call $T^{{\rm sup}}_\alpha$ the \emph{super-standard} extended tableau of shape $\alpha$. In Example~\ref{ex:SET}, $T_1$ is the super-standard extended tableau of shape $(2,1,3)$.

\begin{lemma}\label{lem:cyclicallygenerated}
The module $X_\alpha$ is cyclically generated by $T^{{\rm sup}}_\alpha$.
\end{lemma}
\begin{proof}
Let $S\in \SET(\alpha)$, $S\neq T^{{\rm sup}}_\alpha$. Let $\mathfrak{b}$ be the earliest box of $D(\alpha)$ in which $S$ and  $T^{{\rm sup}}_\alpha$ disagree, where the boxes are ordered by reading rows left to right, starting with the bottom row and proceeding upwards. Suppose $\mathfrak{b}$ has entry $j$ in $S$. Then in $S$, the entry $j-1$ must appear in a later box than $\mathfrak{b}$, and since entries increase along rows and up columns, $j-1$ is strictly above and strictly left of $j$ in $S$. Hence $S = \pi_{j-1}(S')$ for $S'\in \SET(\alpha)$ where $S'$ is $S$ with the entries $j$ and $j-1$ swapped. If the entry ($j-1$) of $\mathfrak{b}$ in $S'$ does not agree with the entry of $\mathfrak{b}$ in  $T^{{\rm sup}}_\alpha$, then repeat the process, resulting in $S''=\pi_{j-2}(S') \in \SET(\alpha)$ that has $j-2$ in $\mathfrak{b}$. Since the entry in $\mathfrak{b}$ decreases by one each time, eventually we obtain $S^*\in \SET(\alpha)$ which agrees with $T^{{\rm sup}}_\alpha$ on all boxes up to and including $\mathfrak{b}$, and $S$ is obtained from $S^*$ via a sequence of the operators. Repeating the process on the next box in which $S^*$ and $T^{{\rm sup}}_\alpha$ disagree, etc, eventually we obtain $S$ from $T^{{\rm sup}}_\alpha$ via a sequence of the operators.
\end{proof}

\begin{lemma}\label{lem:superstandard}
Suppose $T\in SET(\alpha)$ has the property that $\pi_i(T)=T$ for all $i$ such that $i\neq \alpha_1+\cdots + \alpha_r$ for any $1\le r \le \ell(\alpha)$. Then $T=T^{{\rm sup}}_\alpha$.
\end{lemma}
\begin{proof}
The first entry of the first row of $T$ must be $1$, by the increasing row and column conditions. Suppose the first $j$ entries of the first row of $T$ are $1,\ldots , j$ for some $1\le j \le \alpha_1-1$. If $j+1$ is not in the first row of $T$, the increasing row and column conditions force $j+1$ to be weakly left of $j$ and thus $\pi_j(T)\neq T$, contradicting the assumption. Hence the entries of the first row of $T$ are $1, \ldots , \alpha_1$. A similar argument then ensures the entries of the second row of $T$ are $\alpha_1+1, \ldots , \alpha_2$, and continuing thus we obtain $T=T^{{\rm sup}}_\alpha$.
\end{proof}

\begin{theorem}
Let $\alpha\vDash n$. Then $X_\alpha$ is an indecomposable $H_n(0)$-module.
\end{theorem}
\begin{proof}
Let $f$ be an idempotent module endomorphism of $X_\alpha$. We will show $f$ is either zero or the identity, which by \cite[Proposition 3.1]{Jacobson} implies $X_\alpha$ is indecomposable. Suppose
\[f(T^{{\rm sup}}_\alpha) = \sum_{T\in \SET(\alpha)}b_TT.\]
It follows from Lemma~\ref{lem:superstandard} that for any $S\in \SET(\alpha)$ such that $S \neq T^{{\rm sup}}_\alpha$, there exists some $1\le i \le n-1$ such that $\pi_i(T^{{\rm sup}}_\alpha) = T^{{\rm sup}}_\alpha$ but $\pi_i(S)\neq S$.

For such an $i$, we have
 \[f(T^{{\rm sup}}_\alpha) = f(\pi_i(T^{{\rm sup}}_\alpha)) = \pi_if(T^{{\rm sup}}_\alpha) = \pi_i( \sum_{T\in \SET(\alpha)}b_TT) =  \sum_{T\in \SET(\alpha)}b_T\pi_i(T).\]
 The coefficient of $S\neq T^{{\rm sup}}_\alpha$ on the right-hand side of the expression above is zero, since if there was $S'\in \SET(\alpha)$ such that $\pi_i(S')=S$, we would have $\pi_i(S') = \pi_i^2(S') = \pi_i(S) \neq S$, a contradiction. 
 
 Therefore $b_S = 0$ for all $S\neq T^{{\rm sup}}_\alpha$, and we have $f(T^{{\rm sup}}_\alpha) = b_{T^{{\rm sup}}_\alpha}T^{{\rm sup}}_\alpha$. Since $f^2=f$, we must have $b_{T^{{\rm sup}}_\alpha}^2 = b_{T^{{\rm sup}}_\alpha}$, which forces $b_{T^{{\rm sup}}_\alpha}=0$ or $b_{T^{{\rm sup}}_\alpha}=1$. Since $X_\alpha$ is cyclically generated by $T^{{\rm sup}}_\alpha$ (Lemma~\ref{lem:cyclicallygenerated}), we conclude $f$ is either zero or the identity on $X_\alpha$, as required.
\end{proof}

%
%

\bibliographystyle{amsalpha} 
\bibliography{ExtendedSchurModule}

\newcommand{\etalchar}[1]{$^{#1}$}
\providecommand{\bysame}{\leavevmode\hbox to3em{\hrulefill}\thinspace}
\providecommand{\MR}{\relax\ifhmode\unskip\space\fi MR }
\providecommand{\MRhref}[2]{%
  \href{http://www.ams.org/mathscinet-getitem?mr=#1}{#2}
}
\providecommand{\href}[2]{#2}
\begin{thebibliography}{HLMvW11}

\bibitem[AS19]{Assaf.Searles:3}
S.~Assaf and D.~Searles, \emph{{K}ohnert polynomials}, Experiment. Math., to
  appear (2019), 27 pages.

\bibitem[BBS{\etalchar{+}}14]{BBSSZ:2}
C.~Berg, N.~Bergeron, F.~Saliola, L.~Serrano, and M.~Zabrocki, \emph{A lift of
  the {S}chur and {H}all-{L}ittlewood bases to non-commutative symmetric
  functions}, Canad. J. Math. \textbf{66} (2014), no.~3, 525--565.

\bibitem[BBS{\etalchar{+}}15]{BBSSZ}
\bysame, \emph{Indecomposable modules for the dual immaculate basis of
  quasisymmetric functions}, Proc. Amer. Math. Soc. \textbf{143} (2015),
  991--1000.

\bibitem[BLvW11]{BLvW}
C.~Bessenrodt, K.~Luoto, and S.~van Willigenburg, \emph{Skew quasisymmetric
  {S}chur functions and noncommutative {S}chur functions}, Adv. Math.
  \textbf{226} (2011), no.~5, 4492--4532.

\bibitem[Cam16]{Campbell}
J.~Campbell, \emph{Bipieri tableaux}, Australas. J. Combin. \textbf{66} (2016),
  no.~1, 66--103.

\bibitem[CFL{\etalchar{+}}14]{CFLSX}
J.~Campbell, K.~Feldman, J.~Light, P.~Shuldiner, and Y.~Xu, \emph{A
  {S}chur-like basis of {NS}ym defined by a {P}ieri rule}, Electron. J. Combin.
  \textbf{21} (2014), no.~3, Paper 3.41, 19.

\bibitem[DKLT96]{DKLT}
G.~Duchamp, D.~Krob, B.~Leclerc, and J.-Y. Thibon, \emph{Fonctions
  quasi-sym\'etriques, fonctions sym\'etriques non-commutatives, et alg\`ebres
  de {H}ecke \`a $q=0$}, C. R. Math. Acad. Sci. Paris \textbf{322} (1996),
  107--112.

\bibitem[Ges84]{Gessel}
I.~M. Gessel, \emph{Multipartite {$P$}-partitions and inner products of skew
  {S}chur functions}, Combinatorics and algebra ({B}oulder, {C}olo., 1983),
  Contemp. Math., vol.~34, Amer. Math. Soc., Providence, RI, 1984,
  pp.~289--317.

\bibitem[GKL{\etalchar{+}}95]{GKLLRT}
I.~Gelfand, D.~Krob, A.~Lascoux, B.~Leclerc, V.~Retakh, and J.-Y. Thibon,
  \emph{Noncommutative symmetric functions}, Adv. Math. \textbf{112} (1995),
  218--348.

\bibitem[HLMvW11]{HLMvW11:QS}
J.~Haglund, K.~Luoto, S.~Mason, and S.~van Willigenburg, \emph{Quasisymmetric
  {S}chur functions}, J. Combin. Theory Ser. A \textbf{118} (2011), no.~2,
  463--490.

\bibitem[HMR13]{HMR}
J.~Haglund, S.~Mason, and J.~Remmel, \emph{Properties of the nonsymmetric
  {R}obinson-{S}chensted-{K}nuth algorithm}, J. Algebraic Combin. \textbf{38}
  (2013), 285--327.

\bibitem[Jac89]{Jacobson}
N.~Jacobson, \emph{Basic algebra. {II}.}, W. H. Freeman and Company, 1989.

\bibitem[Koh91]{Kohnert}
A.~Kohnert, \emph{Weintrauben, {P}olynome, {T}ableaux}, Bayreuth. Math. Schr.
  (1991), no.~38, 1--97, Dissertation, Universit\"at Bayreuth, Bayreuth, 1990.

\bibitem[K{\"o}n17]{Koenig}
S.~K{\"o}nig, \emph{The decomposition of $0$-{H}ecke modules associated to
  quasisymmetric {S}chur functions}, preprint (2017), 16 pages, {\sf
  arXiv:1711.08737}.

\bibitem[Nor79]{Norton}
P.~N. Norton, \emph{$0$-{H}ecke algebras}, J. Aust. Math. Soc. \textbf{27}
  (1979), no.~3, 337--357.

\bibitem[TvW15]{TvW:1}
V.~Tewari and S.~van Willigenburg, \emph{Modules of the $0$-{H}ecke algebra and
  quasisymmetric {S}chur functions}, Adv. Math. \textbf{285} (2015),
  1025--1065.

\bibitem[TvW19]{TvW:2}
\bysame, \emph{Permuted composition tableaux, $0$-{H}ecke algebra and labeled
  binary trees}, J. Combin. Theory Ser. A \textbf{161} (2019), 420--452.

\end{thebibliography}

\end{document}